\documentclass[reqno]{amsart}
\usepackage{graphicx} 
\usepackage{bbm,amsmath, amsfonts, amsthm, amssymb}
\usepackage{bm}
\usepackage[margin=3.5cm]{geometry}
\usepackage{xcolor}
\usepackage{graphicx,tabularx,enumitem}
\usepackage{mathrsfs}
\usepackage{nicefrac} 
\usepackage{xfrac} 

\usepackage{mathtools}
\graphicspath{{Figures/}}

\usepackage[textsize=tiny]{todonotes}
\newcommand{\EE}{\mathbb{E}}

\newcommand{\NN}{\mathbb{N}}
\newcommand{\PP}{\mathbb{P}}
\newcommand{\RR}{\mathbb{R}}

\newcommand{\D}{\mathrm{d}}
\newcommand{\ds}{\mathrm{d}s}
\newcommand{\dr}{\mathrm{d}r}
\newcommand{\dt}{\mathrm{d}t}
\newcommand{\du}{\mathrm{d}u}

\newcommand{\dx}{\mathrm{d}x}

\newcommand{\E}{\mathrm{e}}

\newcommand{\Ff}{\mathcal{F}}

\newcommand{\Ii}{\mathcal{I}}

\newcommand{\Ll}{\mathcal{L}}

\newcommand{\ep}{\varepsilon}

\newcommand{\half}{\frac{1}{2}}
\newcommand{\one}{\mathbbm{1}}

\newcommand{\abs}[1]{\left\lvert#1\right\rvert}
\newcommand{\norm}[1]{\left\lVert#1\right\rVert}

\newcommand{\notthis}[1]{}
\newtheorem{theorem}{Theorem}[section]
\newtheorem{corollary}[theorem]{Corollary}
\newtheorem{lemma}[theorem]{Lemma}
\newtheorem{proposition}[theorem]{Proposition}

\theoremstyle{definition}
\newtheorem{definition}[theorem]{Definition}
\newtheorem{remark}[theorem]{Remark}

\newtheorem{example}[theorem]{Example}
\theoremstyle{plain}
\numberwithin{equation}{section}

\definecolor{ocean}{rgb}{0,0.1,0.6}

\definecolor{imperialGreen}{RGB}{2,137,59}
\definecolor{imperialBlue}{RGB}{0, 62, 116}
\definecolor{imperialBrick}{RGB}{165,25,0}
\definecolor{imperialProcess}{RGB}{0,133,202}

\usepackage{hyperref}
\hypersetup{colorlinks=false, linkbordercolor=imperialProcess,
citebordercolor=imperialBrick}

\usepackage[foot]{amsaddr}

\title{A BDG inequality for stochastic Volterra integrals}
\author{Alexandre Pannier}
\date{\today}
\address{LPSM, Université Paris Cité}
\email{pannier@lpsm.paris}
\thanks{}

\begin{document}

\begin{abstract}
    We establish Burkholder-Davis-Gundy-type inequalities for stochastic Volterra integrals with a completely monotone convolution kernel, which may exhibit singular behaviour at the origin. When the supremum is taken over a finite interval, the upper bound depends linearly on the $L^\gamma$-norm of the kernel, for any $\gamma>2$. We demonstrate the utility of this inequality in quantifying the pathwise distance between two stochastic Volterra equations with distinct kernels, with a particular emphasis on the multifactor Markovian approximation. For kernels that decay sufficiently fast, we derive an alternative inequality valid over an infinite time interval, providing uniform-in-time bounds for mean-reverting stochastic Volterra equations. Finally, we compare our findings with existing results in the literature.
\end{abstract}

\maketitle

\section{Introduction}
This note investigates a version of the Burkholder-Davis-Gundy (BDG) inequality tailored to stochastic Volterra integrals. Specifically, we examine the following inequality:
\begin{align}\label{eq:BDG_intro}
    \EE\left[\sup_{t\in[0,T]}\abs{\int_0^t K(t-s)\phi(s)\D W_s}^p\right] \le \bm{C} \int_0^T \EE\abs{\phi(s)}^p\ds,
\end{align}
where $p\ge2$, $T>0$ can be finite or infinite, $K\in L^2([0,T];\RR^{d\times d})$, $W$ is an $m$-dimensional Brownian motion, $\bm{C}>0$ is a constant to be determined later and $\phi$ is a previsible proces with values in~$\RR^{d\times m}$ such that the right-hand-side of the inequality is finite. The original BDG inequality, as a consequence of Doob's inequality, applies to local martingales and implies~$\EE[\sup_{t\in[0,T]}\lvert\int_0^t \phi(s)\D W_s\lvert^p]\le \bm{b}_p\EE[(\int_0^T \abs{\phi(s)}^2\ds)^{p/2}]$, where~$\bm{b}_p:=2p^{p/2}$ \cite[Remark 2]{carlen1991p}. Thus it recovers~\eqref{eq:BDG_intro}  when $K\equiv1$, after an application of Jensen's inequality and with the constant~$\bm{C}=T^{p/2-1}\bm{b}_p$. 
This inequality is a fundamental tool in stochastic analysis, playing a central role in deriving a priori estimates and ensuring the well-posedness of Itô stochastic differential equations, with the norm defined by~$\norm{X}_{\mathrm{s}}^p:=\EE[\sup_{t\in[0,T]}\abs{X_t}^p]$. We will now outline the motivation and potential applications that arise from incorporating the kernel~$K$ in the integral.

Stochastic Volterra processes are typically defined as
\begin{align}\label{eq:SVE}
    X_t= x_0(t) + \int_0^t K(t-s)\tilde{b}(s)\ds + \int_0^t K(t-s)\tilde{\sigma}(s)\D W_s, \quad t\ge0,
\end{align}
where the kernel $K$ embeds the intertemporal dependence of the system. Equation \eqref{eq:SVE} is called a Stochastic Volterra Equation (SVE) if there exist measurable functions~$b$ and $\sigma$ such that~$\tilde{b}(s)=b(X_s)$ and $\tilde{\sigma}(s)=\sigma(X_s)$ for all $s\ge0$. 
Stochastic Volterra processes form a widely studied class of models for systems exhibiting memory effects and/or fractional behavior. They gained significant traction in mathematical finance, particularly with the rise of rough volatility models \cite{bayer2023rough}. Beyond finance, these processes have found applications in diverse fields such as electricity price modeling \cite{barndorff2013modelling, bennedsen2022rough}, the study of turbulent flow velocities \cite{barndorff2008stochastic, chevillard2017regularized}, and even climate science \cite{eichinger2020sample}. The applications to rough volatility, in particular, have fueled substantial theoretical advancements, with several key contributions such as \cite{abi2019affine, keller2018affine, viens2019martingale} representing a few notable papers in the field.

Initially, the literature on Stochastic Volterra Equations (SVEs) focused on well-behaved kernels satisfying $K(0)<+\infty$ \cite{protter1985volterra}. However, recent developments have shifted attention to singular kernels for which~$\lim_{t\to0}K(t)=+\infty$. 
The analysis of these more challenging SVEs typically relies on a priori estimates and well-posedness results often using the norm~$\norm{X}_{\mathrm{w}}^p = \sup_{t\in[0,T]}\EE[\abs{X_t}^p]$. This approach is necessary because the classical BDG inequality does not directly apply to processes of the form $(\int_0^t K(t-s)\phi(s)\D W_s)_{t\in[0,T]}$, which  generally lack the local martingale property.  Nonetheless, the BDG inequality still plays a role in this context, allowing to derive the following estimate for $p>2$ (the case $p=2$ being a consequence of Itô's isometry)
\begin{align}\label{eq:BDG_no_sup}
   \EE\abs{\int_0^t K(t-s)\phi(s)\D W_s}^p
     \le \EE\,\left[\,\sup_{r\in[0,t]}\abs{\int_0^r K(t-s)\phi(s)\D W_s}^p\right]\,
     \le \bm{b}_p\EE\abs{\int_0^t \abs{K(t-s)}^2\abs{\phi(s)}^2\ds}^{\frac{p}{2}}.
\end{align}
However, this does not provide a pathwise estimate. The norm $\norm{X}_{\mathrm{s}}$ is recovered a posteriori through Kolmogorov's continuity theorem as demonstrated, for example, in \cite[Lemma 2.4]{abi2019affine}.
Despite being relatively unknown, there are at least two important versions of the BDG inequality for stochastic Volterra integrals that were instrumental in establishing the well-posedness of SVEs. The first, introduced by Decreusefond \cite{decreusefond2002regularity}, was applied to SVEs in \cite{coutin2001stochastic}. The second summarises the Kolmogorov continuity theorem approach \cite[Lemma 3.4]{zhang2010stochastic}. We believe these inequalities merit wider recognition, as they could significantly aid future research. As such, we provide a more detailed exposition of both at the end of this note. Moreover, these inequalities are also valid for non-convolution kernels, meaning they apply when~$K(t-s)$ is replaced with the more general $K(t,s)$ in~\eqref{eq:BDG_intro}. 

In a different setting, mild solutions to stochastic PDEs require a specialised BDG inequality~\cite[Lemma 3.3]{gawarecki2010stochastic} which takes the form of~\eqref{eq:BDG_intro} with the additional assumption that $K$ is a semigroup. Although this version is designed for an infinite-dimensional framework, it indicates a route for dealing with the non-trivial one-dimensional semigroup: the exponential function.

We thus leverage the conducive class of completely monotone kernels. A function $K:(0,\infty)\to[0,\infty)$ is called completely monotone if it is infinitely differentiable on $(0,\infty)$ and satisfies~$(-1)^n \frac{\D^n}{\D t^n}K(t)\ge 0$ for all $n\in\NN\cup\{0\}$ and $t>0$. Crucially, Bernstein's theorem~\cite[Theorem 1.4]{schilling2010bernstein} asserts that this property is equivalent to the existence of a unique non-negative measure~$\mu$ on~$[0,\infty)$ such that~$K(t)=\int_0^\infty \E^{-xt}\mu(\D x)$ for all~$t>0$. The BDG inequality derived in this paper and displayed in Theorem~\ref{thm:BDG_convolution} offers two key advantages:
\begin{enumerate}
    \item[1)] The assumptions are easy to check: $K$ must be completely monotone and belong to~$L^\gamma([0,T])$ for some~$\gamma>2$. This includes a variety of (rough) kernels commonly found in the literature, as illustrated in Example~\ref{ex:kernels}.
    \item[2)] The constant $\bm{C}=C_{p,\gamma,T} \norm{K}_{L^\gamma([0,T])}$ keeps track of the kernel norm.
\end{enumerate}
This second point especially stands in stark contrast with the aforementioned BDG inequalities as well as famous inequalities for the supremum of Gaussian processes such as Borell-TIS, Fernique's theorem or the results of \cite{graversen2000maximal} for Ornstein-Uhlenbeck processes. 
For comparison, usual pointwise estimates (where the supremum is taken outside of the expectation) require~$K\in L^2([0,T])$ and give a constant proportional to this norm, which is only marginally better than~$L^\gamma$. 

Moreover, this result paves the way for important applications pertaining to the comparison of SVEs with different kernels. Indeed, there is a significant interest in approximating SVEs with singular kernels by SVEs that use more regular kernels, which are often easier to study and to simulate as they may be semimartingales and/or Markovian~\cite{carmona1998fractional,abi2019multifactor,alfonsi2024approximation}. These developments are explored further in Section \ref{sec:SVE_comparison} where we present the first pathwise comparison, since only pointwise estimates were available until now. The case of the multifactor approximation is treated specifically in Proposition~\ref{prop:multifactor} and the rates of convergence we obtain are arbitrarily close to the ones derived in~\cite{alfonsi2024approximation} where the supremum is taken outside of the expectation. 

The range of $p$ for which the three BDG inequalities hold depend on the integrability of~$K$. To compare their requirements, consider the well-known power-law kernel~$K(t)=t^{H-\half}$ with~$H\in(0,\half)$. Decreusefond's inequality~\cite[Remark 4.1]{decreusefond2002regularity} and Theorem~\ref{thm:BDG_convolution} necessitate~$p>1/H$, while Zhang's result~\cite[Lemma 3.4]{zhang2010stochastic} only holds for~$p>2/H$. 

In addition, under stronger assumptions, we can take $T=+\infty$ in the BDG inequality~\eqref{eq:BDG_intro}, although this comes at the cost of losing the kernel norm dependence in the constant~$\bm{C}$. This uniform-in-time estimate is presented in Proposition~\ref{prop:BDG_convo_uniform} and is, to the best of our knowledge, the first instance of such an inequality for stochastic Volterra integrals.
The technique is valid provided the kernel decays sufficiently fast and we determine examples where this condition holds. SVEs with linear drift can be reformulated as a stochastic Volterra integral with a different kernel via a variation of the constant formula of Volterra type~\cite[Chapter 2, Theorem 3.5]{gripenberg1990volterra}. As an application, we demonstrate in Corollary~\ref{coro:uniform_time} how uniform-in-time estimates can be derived for such processes. 

To summarise, the contributions of this paper are two inequalities spurring one application each:
\begin{enumerate}
    \item[1)] Case $T<\infty$.
    \begin{enumerate}
        \item[a)] A BDG inequality for completely monotone kernels where~$\mathbf{C}$ is proportional to~$\norm{K}_{L^\gamma([0,T])}$ --- Theorem~\ref{thm:BDG_convolution};
        \item[b)] Pathwise comparisons of SVEs with different kernels --- Corollary~\ref{coro:Comparing_kernels}, and Proposition~\ref{prop:multifactor} for the multifactor approximation.
    \end{enumerate}
    \item[2)] Case $T=\infty$.
    \begin{enumerate}
        \item[a)] A BDG inequality for completely monotone kernels --- Proposition~\ref{prop:BDG_convo_uniform};
        \item[b)] Uniform-in-time estimates for linear SVEs --- Proposition~\ref{coro:uniform_time}.
    \end{enumerate}
\end{enumerate}

The rest of the paper is organised as follows. Section \ref{sec:main_results} presents the main results: the BDG inequalities for $T$ finite and infinite. The proofs of these results are gathered in Section~\ref{sec:proofs} and the applications to SVEs are developped in Section~\ref{sec:applications}. Finally, Section~\ref{sec:nonconv} sheds some light on the pre-existing BDG inequalities for non-convolution kernels. 
\bigskip

\textbf{Notations.} We fix $d,m\in\NN$ and a filtered probability space $(\Omega,\Ff,(\Ff_t)_{t\ge0},\PP)$ satisfying the usual conditions and equipped with an $m$-dimensional Brownian motion~$W$.  The notation~$\abs{\cdot}$ refers to both Euclidean norm in~$\RR^d$ and Frobenius norm in~$\RR^{d\times d}$ or $\RR^{d\times m}$, where $d\in\NN$. For any $\gamma\ge 1$, $T>0$, we denote $\Ll_T^\gamma:=L^\gamma([0,T],\RR^{d\times d})$. For all $p\ge2$, let $\mathbf{b}_p:=2p^{p/2}$ be the BDG constant.

\section{Convolution BDG for completely monotone kernels}\label{sec:main_results}
In this section we study the BDG inequality under the assumption that the kernel is of convolution form. An important subclass thereof is the family of completely monotone kernels. 
\begin{definition}
    We call a kernel~$K:(0,\infty)\to\RR^{d\times d}$ \emph{completely monotone} if there exists a non-negative measure~$\mu$ on~$[0,\infty)$ such that
\begin{align*}
    K(t)=\int_0^\infty \E^{-xt} \mu(\D x)\qquad \text{for all  } t>0.
\end{align*}
\end{definition} 
\begin{theorem}\label{thm:BDG_convolution}
    For some $T>0$ and $\gamma>2$, let $K\in \Ll^\gamma_T$ be a completely monotone kernel.
    Let~$p>\frac{2\gamma}{\gamma-2}$  and consider a previsible process~$\phi$ with values in~$\RR^{d\times m}$ such that $\int_0^T \EE\abs{\phi(s)}^p\ds < \infty$. 
    Then we have
    \begin{align}\label{eq:BDG_convolution}
        \EE\left[\sup_{t\in[0,T]} \abs{\int_0^t K(t-s)\phi(s)\D W_s}^p\right] \le \overline{C}_{p,\gamma,T,d,m} \norm{K}_{\Ll^\gamma_T}^p \, \int_0^T \EE\abs{\phi(s)}^p\ds,
    \end{align}
    where~$\overline{C}_{p,\gamma,T,d,m}=\overline{C}_{p,\gamma}d^{\frac{3p-4}{2}}m^{p-1} T^{p(\half-\frac{1}{\gamma})-1} $ and $\overline{C}_{p,\gamma}>0$ is a constant that depends only on~$p$ and $\gamma$.
\end{theorem}
\begin{remark}
Note that there is a trade-off in choosing smaller~$\gamma$ giving a smaller~$\norm{K}_{\Ll^\gamma_T}$ but restricting to higher~$p$.
\end{remark}


\begin{example}\label{ex:kernels}
We present a few examples of kernels that can be found in the literature and satisfy the assumptions of Theorem~\ref{thm:BDG_convolution}.
\begin{enumerate}
    \item The exponential kernel $K(t)=\E^{-\beta t}$ is completely monotone for all~$\beta\ge0$ with $\mu(\D x)=\delta_\beta(\D x)$ where~$\delta_\beta$ is the Dirac measure at~$\beta$. Moreover, $K\in\Ll^\gamma_T$ for any~$\gamma>2$ hence $p>2$.  When~$\beta=0$ we have~$K\equiv1$ and we recover the classical BDG inequality with the same scaling in time~$T^{p(\half-\frac{1}{\gamma})-1}\norm{K}^p_{\Ll^\gamma_T}=T^{p/2-1}$.
    \item The celebrated power-law kernel~$K(t)=t^{H-\half}$ is completely monotone  for all~$H\in(-\infty,\half)$ with~$\mu(\D x)=\frac{x^{-H-\half}}{\Gamma(\half-H)}\D x$. Moreover, if~$H>0$ then $K\in\Ll^\gamma_T$ for any~$\gamma<\frac{2}{1-2H}$. This entails that~$p$ must be strictly greater than~$1/H$.
    \item The kernel $R_\lambda(t)
        =\lambda t^{H-\half} E_{H+\half,H+\half}(-\lambda t^{H+\half})$, 
    where $E_{\alpha,\beta}(z)=\sum_{n\ge0} \frac{z^n}{\Gamma(\alpha n+\beta)}$ is the Mittag-Leffler function, is completely monotone (see \cite[Appendix A.4]{gorenflo1997fractional} or \cite{tomovski2014laplace} for a full proof) with measure
    \begin{align*}
        \mu_\lambda(\D x)= \frac{1}{\pi}\frac{x^{H+\half}\sin((\pi(H+\half))}{x^{2H+1}+2\lambda x^{H+\half} \cos(\pi(H+\half))+\lambda^2} \dx.
    \end{align*}
    Once again, the BDG inequality holds for all $p>1/H$.
    This kernel naturally arises in various contexts related to Volterra equations, as we exhibit in Section~\ref{sec:SVE_linear}. 
\end{enumerate}
Of particular interest are the following additional examples, where~$K$ is completely monotone with measure~$\mu$:
\begin{enumerate}
    \item[(4)] The exponentially damped kernel $K_{\textrm{exp}}(t):=\E^{-\beta t}K(t)$ is completely monotone for all~$\beta\ge0$ with $\mu_{\textrm{exp}}(\D x)= \one_{x>\beta}\mu(\dx-\beta)$. If moreover $K\in\Ll^\gamma_T$ then so does~$K_{\textrm{exp}}$.
    \item[(5)] The shifted kernel~$K_{\textrm{shift}}(t):=K(t+\ep)$ is bounded and completely monotone for all~$\ep>0$ with measure~$\mu_{\textrm{shift}}(\D x)=\E^{-x\ep}\mu(\D x)$.
\end{enumerate}
Furthermore, the set of completely monotone functions is a convex cone which is closed under multiplication~\cite[Corollary 1.6]{schilling2010bernstein} and pointwise convergence (if~$\lim_{n\to\infty}K_n(t)=K(t)$ for all~$t>0$ and $(K_n)_{n\in\NN}$ is a sequence of completely monotone functions then so is~$K$) \cite[Corollary 1.7]{schilling2010bernstein}. 
It is clear that multidimensional versions of these examples also satisfy the necessary conditions. On the other hand, the regular versions with $H\ge\half$ are not covered; this is not really an issue since the associated stochastic Volterra integral is a semimartingale in that case and the standard BDG inequality applies to the local martingale term.
\end{example}
Under certain assumptions, we can take the limit as $T$ goes to $+\infty$ and obtain a uniform-in-time estimate for the moments of the stochastic Volterra integral. Essentially the kernel needs to decay fast enough to control the growth of the integral as~$T$ increases.
\begin{proposition}\label{prop:BDG_convo_uniform}
 Consider $p\ge2$, a previsible process $\phi$ with values in $\RR^{d\times m}$ such that~$\int_0^\infty \EE\abs{\phi(s)}^p\ds<\infty$
and a completely monotone kernel $K:\RR_+\to\RR^{d\times d}$ with measure $\mu$ such that
\begin{align}\label{eq:condition_infty}
    M_p := \bigg\lvert\int_0^\infty x^{\frac{2-p}{2p}}\mu(\D x)\bigg\lvert <\infty.
\end{align}
\textbf{1)} The following inequality holds
    \begin{align}\label{eq:BDG_convo_uniform}
        \EE\left[\sup_{t\in[0,+\infty)}\abs{\int_0^t K(t-s)\phi(s)\D W_s}^p\right] 
        \le C_{p,d,m} M_p^p \int_0^\infty \EE\abs{\phi(s)}^p\ds ,
    \end{align}
    where $C_{p,d,m}= d^{\frac{3p-4}{2}}m^{p-1} \mathbf{b}_p \Gamma\left(\frac{p-2}{2p^2}\right)^{p/2} \Gamma\left(\frac{p-2}{2p}\right)^{p-1} $.\\
\textbf{2)} Furthermore, for $d=1$ the condition \eqref{eq:condition_infty} is satisfied for the following kernels and values of~$p$:
\begin{enumerate}
    \item[a)] $K_{\beta,H}(t)=\E^{-\beta t}t^{H-\half}$ for all~$t>0$, where~$\beta>0$ and~$H\in(0,\half]$, and with~$p>1/H$. 
    \item[b)]  $R_\lambda(t)
        =\lambda t^{H-\half} E_{H+\half,H+\half}(-\lambda t^{H+\half})$, 
    where $H\in(0,\half)$, $\lambda>0$ and $E_{\alpha,\beta}(z)=\sum_{n\ge0} \frac{z^n}{\Gamma(\alpha n+\beta)}$ is the Mittag-Leffler function, and with $p>1/H$.
\end{enumerate}
\end{proposition}
\section{Proofs of the main results}\label{sec:proofs} 
Both of our main results rely on the following lemma, which provides an initial step in the computations.
\begin{lemma}\label{lemma:BDG_conv}
Let~$K\in \Ll^2_T$ be a completely monotone kernel.
    Let~$p\ge2$  and consider a previsible process~$\phi$ with values in~$\RR^{d\times m}$ such that $\int_0^T \EE\abs{\phi(s)}^p\ds < \infty$. 
    Then we have
    \begin{align}\label{eq:lemma_BDG_convolution}
        &\EE\left[\sup_{t\in[0,T]} \abs{\int_0^t K(t-s)\phi(s)\D W_s}^p\right] \\
        &\le d^{\frac{3p-4}{2}}m^{p-1}\mathbf{b}_p \bigg\lvert\int_0^T s^{(\alpha-1)\frac{p}{p-1}} \abs{K(s)}\ds\bigg\lvert^{p-1} \bigg\lvert\int_0^\infty \left(\int_0^T \E^{-2xs}s^{-2\alpha} \ds\right)^{p/2\!\!\!}\mu(\D x)\bigg\lvert \int_0^T \EE\abs{\phi(s)}^p\ds,\nonumber
    \end{align}
    for any $\alpha\in(0,1/2)$ such that the right-hand-side is finite.
\end{lemma}
\begin{proof}

This proof is inspired from the the BDG inequality for stochastic integrals with a semigroup, that appear in mild solutions to SPDEs, see e.g. \cite[Lemma 3.3]{gawarecki2010stochastic}. 
We treat the case~$d=m=1$ and the multidimensional case follows by studying it componentwise since
\begin{align*}
    \bigg\vert\int_0^t K(t-s)\phi(s)\D W_s\bigg\vert^p
    &=\left(\sum_{i=1}^d \bigg\vert\sum_{k=1}^d\sum_{j=1}^m \int_0^t K^{ik}(t-s) \phi^{kj}(s) \D W^j_s\bigg\vert^2\right)^{p/2}\\
    &\le d^{\frac{3p-4}{2}} m^{p-1} \sum_{i,j,k} \abs{ \int_0^t K^{ik}(t-s) \phi^{kj}(s) \D W^j_s}^p.
\end{align*}
First recall that for any~$0\le s<t\le T$, and all~$\alpha\in(0,1)$,
\begin{align*}
    \int_s^t (t-u)^{\alpha-1} (u-s)^{-\alpha}\du =\Gamma(\alpha)\Gamma(1-\alpha)= \frac{\pi}{\sin(\pi \alpha)}=:C_\alpha^{-1}\ge 1.
\end{align*}
The stochastic Fubini theorem~\cite[Theorem 65]{protter2005stochastic} entails (as $\E^{-xt}\mu(\D x)$ is a finite measure for all $t>0$)
\begin{align*}
    \int_0^t K(t-s) \phi(s) \D W_s
    &=C_\alpha\int_0^t \left(\int_s^t (t-u)^{\alpha-1} (u-s)^{-\alpha}\du \right) \left(\int_0^\infty \E^{-x(t-s)} \mu(\D x)\right) \phi(s) \D W_s\\
    &= C_\alpha \int_0^t (t-u)^{\alpha-1}  \int_0^\infty \E^{-x(t-u)}\left(\int_0^u\E^{-x(u-s)} (u-s)^{-\alpha}\phi(s)\D W_s \right)\mu(\D x)\du.
\end{align*}
We introduce, for all~$u\in[0,T]$ and $x\ge0$, the random field
\begin{align*}
    Y(u,x):=\int_0^u\E^{-x(u-s)} (u-s)^{-\alpha}\phi(s)\D W_s .
\end{align*}
For any $p\ge2$, H\"older's inequality yields 

\begin{align}
    \abs{\int_0^t K(t-s)\phi(s)\D W_s}^{p}
    &=C_\alpha^{p} \abs{\int_0^t \int_0^\infty (t-u)^{\alpha-1}\E^{-x(t-u)} Y(u,x)\mu(\D x)\du}^{p} \nonumber\\
    &\le \left(\int_0^t\int_0^\infty (t-u)^{(\alpha-1)\frac{p}{p-1}}\E^{-x(t-u)\frac{p}{p-1}}\mu(\D x) \du \right)^{p-1} \int_0^t \int_0^\infty \abs{Y(u,x)}^p\mu(\D x) \du \nonumber \\
    &\le \abs{\int_0^T u^{(\alpha-1)\frac{p}{p-1}} K(u) \du }^{p-1} \int_0^t \int_0^\infty \abs{Y(u,x)}^p\mu(\D x) \du. \label{eq:ineq_mu_positive_1}
\end{align}
We apply BDG (without supremum, as in \eqref{eq:BDG_no_sup}) and Young's convolution inequalities to obtain
\begin{align*}
    \EE\left[\int_0^T \int_0^\infty \abs{Y(u,x)}^p\mu(\D x) \du\right] 
    &\le \mathbf{b}_p \EE\int_0^\infty \int_0^T  \left(\int_0^u\E^{-2x(u-s)}(u-s)^{-2\alpha} \phi(s)^2\ds\right)^{p/2} \du \,\mu(\D x)\\
    &\le  \mathbf{b}_p \int_0^\infty \left(\int_0^T \E^{-2xs} s^{-2\alpha} \ds\right)^{p/2} \!\!\mu(\D x)\, \EE\int_0^T \abs{\phi(s)}^p\ds.
\end{align*}
This yields the claim.
\end{proof}

\begin{proof}[Proof of Theorem \ref{thm:BDG_convolution}]
    The proof consists in giving more explicit bounds to the integrals appearing in~\eqref{eq:lemma_BDG_convolution}. Once again we only provide it for the one-dimensional case.  
     Set~$\alpha\in(\frac{p+\gamma-1}{p\gamma},\frac{p\gamma-2}{2p\gamma})$ which is a non-empty interval for all~$p>\frac{2\gamma}{\gamma-2}$, and where $\gamma>2$.
H\"older's inequality entails
\begin{align*}
     \left(\int_0^T s^{(\alpha-1)\frac{p}{p-1}} K(s) \ds \right)^{p-1}
     \le \left(\int_0^T s^{(\alpha-1)\frac{p}{p-1}\frac{\gamma}{\gamma-1}}\ds \right)^{\frac{(p-1)(\gamma-1)}{\gamma}} \norm{K}_{\Ll^\gamma_T}^{p-1},
\end{align*}
where the integral on the right-hand-side is finite thanks to the lower bound~$\alpha>\frac{p+\gamma-1}{p\gamma}$. 
By Minkowski's integral inequality and H\"older's inequality we get
\begin{align*}
    &\left(\int_0^\infty \left(\int_0^T \E^{-2xs} s^{-2\alpha} \ds\right)^{p/2} \mu(\D x) \right)^{2/p}
    \le \int_0^T\left(\int_0^\infty \E^{-pxs} s^{-p\alpha}\mu(\D x)\right)^{2/p}\ds \\
    &\le \int_0^T K(s)^{2/p}s^{-2\alpha}\ds
    \le \norm{K}_{\Ll^\gamma_T}^{2/p} \left(\int_0^T s^{-2\alpha\frac{p\gamma}{p\gamma-2}}\ds \right)^{1-\frac{2}{p\gamma}}.
\end{align*}
where this last integral is finite thanks to the upper bound~$\alpha<\frac{p\gamma-2}{2p\gamma}$.
Regarding the constants, we note that
\begin{align*}
    \left(\int_0^T u^{(\alpha-1)\frac{p}{p-1}\frac{\gamma}{\gamma-1}}\du \right)^{\frac{(p-1)(\gamma-1)}{\gamma}} \left(\int_0^T s^{-2\alpha\frac{p\gamma}{p\gamma-2}}\ds \right)^{(1-\frac{2}{p\gamma})\frac{p}{2}}
    =: C_{\alpha,p,\gamma} T^{p(\half-\frac{1}{\gamma})-1}.
\end{align*}
 We can choose any~$\alpha\in(\frac{p+\gamma-1}{p\gamma},\frac{p\gamma-2}{2p\gamma})$, hence we define
$\overline{C}_{p,\gamma}$ te be higher than $\mathbf{b}_p C_{\alpha,p,\gamma}$ for some~$\alpha$ in this interval. 
Gathering our estimates together yields
\begin{align*}
    \EE\left[\sup_{t\in[0,T]} \abs{\int_0^t K(t-s)\phi(s)\D W_s}^p\right] 
    \le \overline{C}_{p,\gamma} T^{p(\half-\frac{1}{\gamma})-1}  \norm{K}_{\Ll^\gamma_T}^{p} \int_0^T \EE\abs{\phi(s)}^p\ds,
\end{align*}
which concludes the proof.
\end{proof}
\begin{proof}[Proof of Proposition \ref{prop:BDG_convo_uniform}]
    \textbf{(1)} We provide the proof in the one-dimensional case again. Taking limits as~$T\to\infty$ on both sides of \eqref{eq:lemma_BDG_convolution}  and applying the monotone convergence theorem   shows that
    \begin{align} \label{eq:proof_infinity}
    &\EE\left[\sup_{t\in[0,\infty)} \abs{\int_0^t K(t-s)\phi(s)\D W_s}^p\right] \nonumber\\
        &\le \mathbf{b}_p\left(\int_0^\infty s^{(\alpha-1)\frac{p}{p-1}} K(s)\ds\right)^{p-1} \int_0^\infty \left(\int_0^\infty \E^{-2xs}s^{-2\alpha} \ds\right)^{p/2}\!\!\!\mu(\D x) \int_0^\infty \EE\abs{\phi(s)}^p\ds,
    \end{align}
    for any $\alpha\in(0,1)$ such that the right-hand-side is finite. It is thus a matter of proving the latter. 
    Since $K(t)=\int_0^\infty \E^{-xt}\mu(\D x)$, and setting~$p>1/\alpha$ and $\alpha\in(0,\half)$, we are led to study the following integrals 
    \begin{align*}
        \int_0^\infty s^{(\alpha-1)\frac{p}{p-1}} K(s)\ds 
        = \int_0^\infty \int_0^\infty s^{(\alpha-1)\frac{p}{p-1}} \E^{-xs} \ds \mu(\D x)
        = \Gamma\left(\frac{\alpha p-1}{p-1}\right) \int_0^\infty x^{\frac{1-\alpha p}{p-1}} \mu(\D x),
    \end{align*}
    where we applied Fubini-Tonelli's theorem. For the second integral, notice that for any $\alpha<1/2$ we have~$\int_0^\infty \E^{-2xs}s^{-2\alpha} \ds
        = (2x)^{2\alpha-1} \Gamma(1-2\alpha)$, and hence 
    \begin{align}\label{eq:bound_int_infty}
        \int_0^\infty \left(\int_0^\infty \E^{-2xs}s^{-2\alpha} \ds\right)^{p/2}\!\!\!\mu(\D x) 
        \le \Gamma(1-2\alpha)^{p/2} \int_0^\infty x^{p(\alpha-\half)}\mu(\D x).
    \end{align}
    Setting $\alpha=\frac{2-p+p^2}{2p^2}$ we get $p>1/\alpha$; $\frac{1-\alpha p}{p-1}=p(\alpha-\half)=\frac{2-p}{2p}$
    and $1-2\alpha=\frac{p-2}{2p^2}$.
    We are left to check that $\int_0^\infty x^\frac{2-p}{2p}\mu(\D x)$ is finite for the kernels introduced in the Proposition. 

    \textbf{(2a)}
    The kernel defined by~$K_{\beta,H}(t)=t^{H-\half}\E^{-\beta t}$ for all~$t>0$ with~$\beta>0,H\in(0,\half]$ is completely monotone with measure
    \begin{align*}
        \mu_{\beta,H}(\D x)= \delta_{\beta}(\D x) \one_{H=\half} + \Gamma(1/2-H)^{-1} (x-\beta)^{-H-\half} \one_{x>\beta}\one_{H<\half}.
    \end{align*}
    In the case~$H=\half$ we have $\int_0^\infty x^\frac{2-p}{2p}\delta_{\beta}(\D x)=\beta^{\frac{2-p}{2p}}<\infty$. On the other hand, for~$H\in(0,\half)$ and $p>1/H$, the quantity of interest reads~\cite[3.191-2]{gradshteyn2014table}
    \begin{align*}
        \int_{\beta}^\infty x^{\frac{2-p}{2p}} (x-\beta)^{-H-\half} \dx  
        =\beta_i^{1/p-H}\mathrm{B}\left(H-1/p,1/2-H\right)<\infty,
    \end{align*}
    where $\mathrm{B}(a,b)=\int_0^1 x^{a-1} (1-x)^{b-1}\dx$ denotes the Beta function. 
    
    \textbf{(2b)} The kernel $R_\lambda(t)
        =\lambda t^{H-\half} E_{H+\half,H+\half}(-\lambda t^{H+\half})$, is completely monotone (see \cite[Appendix A.4]{gorenflo1997fractional} or \cite{tomovski2014laplace} for a full proof) with measure
    \begin{align*}
        \mu_\lambda(\D x)= \frac{1}{\pi}\frac{x^{H+\half}\sin((\pi(H+\half))}{x^{2H+1}+2\lambda x^{H+\half} \cos(\pi(H+\half))+\lambda^2} \dx
        \le \frac{x^{H+\half}}{\lambda(\sin(\pi(H+\half))-\cos(\pi(H+\half)))} \dx.
    \end{align*}
    Since the polynomial in the denominator is greater than the positive constant~$C_{\lambda,H}:=\lambda(\sin(\pi(H+\half))-\cos(\pi(H+\half)))$, we have for any~$p>2$ and $N>0$ that
    \begin{align*}
        \int_0^N x^{\frac{2-p}{2p}}\mu_\lambda(\D x)\le C_{\lambda,H}^{-1} \int_0^N x^{\frac{2-p}{2p}+H+\half}\dx
        = C_{\lambda,H}^{-1} \int_0^N x^{1/p+H}\dx<\infty.
    \end{align*}
    On the other hand, for $N$ large enough and $x\ge N$, it holds $\mu_\lambda(\D x)\le 2x^{-H-\half}\dx$ and thus
    \begin{align*}
        \int_N^\infty x^{\frac{2-p}{2p}}\mu_\lambda(\D x)\le 2 \int_N^\infty x^{\frac{2-p}{2p}-H-\half}\dx 
        =2\int_N^\infty x^{1/p-H-1}\dx, 
    \end{align*}
    which is finite for any~$p>1/H$.
\end{proof}
\section{Applications to stochastic Volterra equations}\label{sec:applications}
As we hinted at in the introduction, the main interest in studying stochastic Volterra integrals lies in its interplay with SVEs. We explore in this section the applications of the BDG inequalities~\ref{eq:BDG_convolution} and \ref{eq:BDG_convo_uniform} to this class of equations.

\subsection{Comparing SVEs with different kernels}\label{sec:SVE_comparison}
Let $d=m=1$ in this section. Consider a Lipschitz continuous function $\sigma:\RR\to\RR$ with linear growth and two kernels~$K_1,K_2\in \Ll^\gamma_T$  for some~$\gamma>2$ such that~$K_2$ and~$K_1-K_2$ (or $K_2-K_1$) are completely monotone. For a continous function~$x_0:[0,T]\to\RR$, introduce  the stochastic Volterra equations
\begin{align}\label{eq:SVE_comp}
    &X_t=x_0(t) + \int_0^t K_1(t-s)\sigma(X_s)\D W_s,\\
    &Y_t = x_0(t) + \int_0^t K_2(t-s) \sigma(Y_s)\D W_s, \quad t\in[0,T].\nonumber
\end{align}
Standard results, see e.g. \cite[Theorem 3.1]{zhang2010stochastic}, show that they both have a unique solution and~$\sup_{t\in[0,T]}\EE[\abs{X_t}^p+\abs{Y_t}^p] <\infty$.
\begin{corollary}
    \label{coro:Comparing_kernels}
    For all~$p>0$, there is a constant~$C>0$ depending on~$p,\gamma,T$ such that
    \begin{align}
        \EE\left[\sup_{t\in[0,T]} \abs{X_t-Y_t}^p\right]\le C \norm{K_1-K_2}_{\Ll^\gamma_T}^p.
    \end{align}
\end{corollary}
\begin{proof}
In this proof the constant~$C>0$ may change from line to line. 
For any~$p>\frac{2\gamma}{\gamma-2}$, the BDG inequality~\eqref{eq:BDG_convolution} yields 
\begin{align*}
    \EE\left[\sup_{t\in[0,T]} \abs{X_t - Y_t}^{p} \right] 
&\le \EE\left[\sup_{t\in[0,T]}\abs{\int_0^t \big( K_1(t-s) - K_2(t-s)\big) \sigma(X_s)\D W_s}^p\right]\\
&\qquad+ \EE\left[\sup_{t\in[0,T]}\abs{\int_0^t K_2(t-s) \big( \sigma(X_s)-\sigma(Y_s)\big) \D W_s}^p\right] \\
&\le C \norm{K_1-K_2}_{\Ll^\gamma_T}^p \int_0^T \EE\abs{\sigma(X_s)}^p \ds + C\norm{K_2}_{\Ll^\gamma_T}^p \int_0^T \EE\abs{\sigma(X_s)-\sigma(Y_s)}^p\ds \\
&\le C \norm{K_1-K_2}_{\Ll^\gamma_T}^p + C \int_0^T \EE\left[\sup_{t\in[0,s]}\abs{X_t-Y_t}^p\right]\ds,
\end{align*}
where we also used the Lipschitz continuity and linear growth of~$\sigma$.
Gr\"onwall's inequality thus shows that 
\begin{align*}
    \EE\left[\sup_{t\in[0,T]} \abs{X_t - Y_t}^{p} \right] 
&\le C \norm{K_1-K_2}_{\Ll^\gamma_T}^p.
\end{align*}
Moreover, for any $q\in(0,p)$ we have
\begin{align*}
    \EE\left[\sup_{t\in[0,T]} \abs{X_t - Y_t}^{q} \right] 
    \le 
    \EE\left[\sup_{t\in[0,T]} \abs{X_t - Y_t}^{p} \right]^{q/p}  
\le C^{q/p} \norm{K_1-K_2}_{\Ll^\gamma_T}^q.
\end{align*}
This concludes the proof.
\end{proof}

\subsubsection{Shifted kernel approximation}
Any completely monotone kernel $K$ with measure $\mu$ can be approximated by its shifted version~$K^\ep:t\to K(t+\ep)$, for any $\ep>0$. This may allow to avoid the singularity as $t\to0$ when considering singular kernels, which are prominent in the rough volatility literature. Moreover $K^\ep$ is completely monotone with measure $\mu^\ep(\D x)=\E^{-x\ep}\mu(\D x)$ since~$ K^\ep(t) = K(t+\ep)= \int_0^\infty \E^{-x(t+\ep)}\mu(\D x)$, see Example~\ref{ex:kernels}. 
Hence $K-K^\ep$ is completely monotone with the positive measure $\mu-\mu^\ep$ and Corollary \ref{coro:Comparing_kernels} then indicates that the distance between the SVEs associated to $K$ and $K^\ep$ is controlled by $\norm{K-K^\ep}_{\Ll^\gamma_T}^p$. 

Let us take as a working example the power-law kernel~$K(t)=t^{H-\half}$ and its shifted version $K^\ep(t)=(t+\ep)^{H-\half}$ for any $t\in[0,T]$. Then for any~$\gamma\in(2,\frac{2}{1-2H})$ convexity arguments yield
\begin{align*}
    \norm{K-K^\ep}_{\Ll^\gamma_T}^{\gamma}
    &=\int_0^T \left(s^{H-\half}-(s+\ep)^{H-\half}\right)^{\gamma} \ds
    \le \int_0^T \left(s^{\gamma(H-\half)}-(s+\ep)^{\gamma(H-\half)}\right) \ds\\
    &= \frac{(T+\ep)^{\gamma(H-\half)+1}- T^{\gamma(H-\half)+1} + \ep^{\gamma(H-\half)+1}}{\gamma(H-\half)+1}
    \le \frac{\ep^{\gamma(H-\half)+1}}{\gamma(H-\half)+1}.
\end{align*}
The rate of convergence can thus be anything short of~$H$; indeed for any~$\delta\in(0,H)$ one can set~$\gamma=\frac{2}{1-2\delta}$ and obtain 
\begin{align*}
    \norm{K-K^\ep}_{\Ll^\gamma_T}\le  
    (H-\delta)^{\delta-\half}\, \ep^{H-\delta} \le \frac{\ep^{H-\delta}}{\sqrt{H-\delta}}.
\end{align*}
\subsubsection{Multifactor approximation}\label{sec:SVE_comparison_multifactor}
An important stream of the literature about SVEs is concerned with multifactor approximations. They consist in approximating~$X$, solution to~\eqref{eq:SVE_comp} with the (rough) kernel $K(t)=\int_0^\infty \E^{-xt}\mu(\D x)$ by $\widehat{X}^{N,n}$, solution to~\eqref{eq:SVE_comp} with the discretised (and regular) version $\widehat{K}^{N,n}=\sum_{i=1}^n w_i \E^{-x_i t}$ for well-chosen weights and nodes~$(w_i,x_i)_{i=1}^n$. 

The approximation $\widehat{X}^{N,n}$ is in fact Markovian in $\RR^n$, hence the name multifactor approximation. Several authors studied the (rate of) convergence~$\EE\lvert X_t-\widehat{X}^{N,n}_t\lvert^2$ of such approximations~\cite{abi2019multifactor,alfonsi2024approximation,bayer2023markovian}. To fix ideas we set $n\in\NN$ points over $[0,N]$ with $N>0$ and for all~$i=1,\cdots,n$ we let~$x_{i}\in\left[u_{i-1},u_i\right)$ and~$w_i:=\mu\left(\left[u_{i-1},u_i\right)\right)$, where $u_i:=\frac{iN}{n}$. This choice of discretisation is made here for the simplicity of the exposition but it is not optimal and we leave this question for future research.

We proceed in two steps, hence we first introduce the process~$X^N$ which solves SVE~\eqref{eq:SVE_comp} with the truncated kernel~$K^N$:
\begin{align*}
    K^N(t) = \int_0^N \E^{-xt}\mu(\D x)=: \int_0^\infty \E^{-xt} \mu_N(\D x), \quad\text{for any } N>0.
\end{align*}
Based on Corollary~\ref{coro:Comparing_kernels} we present a way of estimating the pathwise distance.
\begin{proposition}\label{prop:multifactor}
    Let $K\in\Ll^\gamma_T$ for some $\gamma>2$ be completely monotone and $\sigma:\RR\to\RR$ be a bounded and Lipschitz continuous function. Then for any $p\ge1$ there is a constant $C>0$ independent of $(x_i,w_i)_{i=1}^n, n,N$ such that
    \begin{align}\label{eq:comp_multifactor1}
        &\EE\left[\sup_{t\in[0,T]}\big\lvert X_t-X^N_t\big\lvert^p\right]^{1/p}
        \le C \int_N^\infty x^{-\frac{1}{\gamma}} \mu(\D x);\\
        &\label{eq:comp_multifactor2}
        \EE\left[\sup_{t\in[0,T]}\big\lvert X^N_t-\widehat{X}^{N,n}_t\big\lvert^p\right]^{1/p}
        \le C\,\frac{\mu([0,N)) N}{n}.
    \end{align}
\end{proposition}
\begin{example}
Let~$\mu$ be such that~$\mu(\D x)\le C_\mu x^{-H-\half}\dx$ for some~$H\in(0,\half)$ and constant~$C_\mu>0$. Then we have $\mu([0,N))\le C_\mu N^{\half-H}/(\half-H)$ and  setting~$\gamma=\frac{2}{1-2\delta}$ with~$\delta<H$ we obtain
\begin{align*}
    \int_N^\infty x^{-\frac{1}{\gamma}} \mu(\D x)\le C_\mu \int_N^\infty  x^{-H-\half-\frac{1}{\gamma}}\dx
    \le C_\mu \frac{N^{\delta-H}}{H-\delta}.
\end{align*}
Such an example of kernel is the Gamma kernel~$K(t)=\E^{-\beta t} t^{H-\half}$ with $H\in(0,\half),\beta\ge0$ for which~$\mu(\D x)=\frac{1}{\Gamma(\half-H)} (x-\beta)^{-H-\half}  \one_{x>\beta}\dx$. We can compare the results with the ones obtained in~\cite{alfonsi2024approximation} which do not include the supremum (and with~$p=2$). The analogue to~\eqref{eq:comp_multifactor1} is found in~\cite[Lemma 3.1]{alfonsi2024approximation} and yields a speed proportional to~$N^{-H}$ while the second error~\eqref{eq:comp_multifactor2} is identical to the version without supremum from~\cite[Corollary 3.1]{alfonsi2024approximation}. 
\end{example}
\begin{proof}
The truncated kernel~$K^N$ is defined such that $K-K^N$ is also completely monotone with the non-negative measure~$\mu-\mu_N$. Hence by Corollary~\ref{coro:Comparing_kernels}, for any $p>0$, the distance~$\EE\big[\sup_{t\in[0,T]}\abs{X_t-X^N_t}^p\big]^{1/p}$ is controlled by
\begin{align*}
    \norm{K-K^N}_{\Ll^\gamma_T}
    =\norm{\int_N^\infty \E^{-x\cdot} \mu(\D x)}_{\Ll^\gamma_T}
    \le \int_N^\infty \norm{\E^{-x\cdot}}_{\Ll^\gamma_T} \mu(\D x)
    &= \int_N^\infty \left(\frac{1-\E^{\gamma xT}}{\gamma x}\right)^{\frac{1}{\gamma}} \mu(\D x).
\end{align*}
In the second step we approximate the truncated integral with the discretised one:
\begin{align*}
    \widehat{K}^{N,n}(t) = \sum_{i=1}^n w_i \E^{-x_{i} t}=\int_0^\infty \E^{-xt}\hat\mu_{N,n}(\D x),
\end{align*}
where $\hat\mu_{N,n}(\D x) = \sum_{i=1}^n w_i \delta_{x_{i}}(\D x)$.
Notice that~$\widehat{K}^{N,n}$ is completely monotone but, unfortunately, neither~$K^N-\widehat{K}^{N,n}$ nor $K^N-\widehat{K}^{N,n}$ is. We thus have to resort to a different type of analysis. For the remainder of the proof, $C>0$ will be a constant independent of $(x_i,w_i)_{i=1}^n, n,N$ that may change from line to line. We also set~$p>\frac{2\gamma}{\gamma-2}$ with~$\gamma$ such that~$K\in\Ll^\gamma_T$ (and hence~$K^N\in\Ll^\gamma_T$).
As in the proof of Corollary~\ref{coro:Comparing_kernels}, we obtain thanks to the BDG inequality~\eqref{eq:BDG_convolution}
\begin{align}
    \EE\left[\sup_{t\in[0,T]} \abs{X^N_t - \widehat{X}^{N,n}_t}^{p} \right] 
&\le \EE\left[\sup_{t\in[0,T]}\abs{\int_0^t \big( \widehat{K}^{N,n}(t-s) - K^N(t-s)\big) \sigma(\widehat{X}^{N,n}_s)\D W_s}^p\right] \nonumber\\
&\qquad+ C\norm{K^{N}}_{\Ll^\gamma_T}^p \int_0^T \EE\abs{X^N_s-\widehat{X}^{N,n}_s}^p\ds.\label{eq:diff_XN}
\end{align}
For each $x>0$, define $U^x_t=\int_0^t \E^{-x(t-s)}\sigma(\widehat{X}^{N,n}_s)\D W_s$. In this way the stochastic Fubini theorem and Jensen's inequality entail
\begin{align}\label{eq:diff_KN}
    &\abs{\int_0^t \big( \widehat{K}^{N,n}(t-s) - K^N(t-s)\big) \sigma(\widehat{X}^{N,n}_s)\D W_s}^p
    = \abs{\sum_{i=1}^n \int_{u_{i-1}}^{u_i} \big(U^{x_i}_t - U^x_t\big) \mu(\D x)}^p \nonumber\\
    &=\abs{ \int_{0}^{N} \big(U^{\sum_{i=1}^n x_i \one_{x\in[u_{i-1},u_i)}}_t - U^x_t\big) \mu(\D x)}^p 
    \le \mu([0,N))^{p-1}  \sum_{i=1}^n \int_{u_{i-1}}^{u_i} \abs{U^{x_i}_t - U^x_t}^p \mu(\D x).
\end{align}
Since $U^x$ is solution to the SDE $\D U^x_t=-xU^x_t\dt+\sigma(\widehat{X}^{N,n}_t)\D W_t$ with $U^x_0=0$, the equation~$\D(U^{x_i}_t - U^x_t) = -x_i(U^{x_i}_t-U^x_t)\dt +(x-x_i) U^x_t\dt$ holds almost surely, which solution can be expressed as~$U^{x_i}_t-U^x_t=\int_0^T \E^{-x_i(t-s)} (x-x_i)U^x_s\ds$. For any $\gamma>2$, the Volterra BDG inequality of Theorem~\ref{thm:BDG_convolution} yields 
\begin{align}\label{eq:diff_Ux}
    \EE\left[\sup_{t\in[0,T]}\abs{U^{x_i}_t-U^x_t}^p\right]
    &\le C \abs{x_i-x}^p\left(\int_0^T \E^{-\gamma x_i s}\ds\right)^{p/\gamma}  \int_0^T \EE\abs{U^x_s}^p\ds.
\end{align}
By the classical BDG and Jensen's inequalities we have
\begin{align*}
    \EE\abs{U^x_t}^p
    \le \bm{b}_p\EE \left(\int_0^t \E^{-2x(t-s)}\sigma(\widehat{X}^{N,n}_s)^2\ds\right)^{p/2} \le \bm{b}_p  \norm{\sigma}^p_\infty T^{p/2-1}.
\end{align*}
Noticing that~$\abs{x-x_i}\le N/n$, Equation \eqref{eq:diff_Ux} entails 
\begin{align}\label{eq:sum_diff_Ux}
    \sum_{i=1}^n \int_{u_{i-1}}^{u_i} \EE\left[\sup_{t\in[0,T]}\abs{U^{x_i}_t - U^x_t}^p\right]\mu(\D x)
    \le C \mu([0,N)) \left(\frac{N}{n}\right)^p.
\end{align}
We conclude by combining the estimates from Equations \eqref{eq:diff_XN}, \eqref{eq:diff_KN} and \eqref{eq:sum_diff_Ux} with Gr\"onwall's lemma:
\begin{align*}
    \EE\left[\sup_{t\in[0,T]} \abs{X^N_t - \widehat{X}^{N,n}_t}^{p} \right] 
\le \left(\frac{\mu([0,N)) N}{n}\right)^p  \E^{C\norm{K^{N}}_{\Ll^\gamma_T}^p}.
\end{align*}
The constant is obtained from the inequality $K^N\le K$ and the case~$p\le 2$ via Jensen's inequality.
\end{proof}

\subsection{SVEs with linear drift}\label{sec:SVE_linear}
In this section we consider a mean-reverting type of SVE with linear drift
\begin{align*}
    X_t = x_0(t) - \lambda\int_0^t K(t-s)X_s\ds + \int_0^t K(t-s)\phi(s)\D W_s,
\end{align*}
where $\lambda>0$, $x_0:\RR_+\to\RR^d$ is a continuous function and $\phi:\RR_+\to\RR^{d\times m}$ is a previsible process. Moreover, for any $T>0$ we assume that $K\in \Ll^2_T$ is such that $\lambda K$ has a resolvent, that is 
a function~$R_\lambda:[0,T]\to\RR^{d\times d}$ such that~$\lambda K -R_\lambda = \lambda K\ast R_\lambda$ where $\ast$ denotes the convolution (we refer to \cite[Chapter 2]{gripenberg1990volterra} for more details). If such a resolvent exists then it is unique.
The variation of constants formula of Volterra type~\cite[Chapter 2, Theorem 3.5]{gripenberg1990volterra} allows to express $X_t$ as
\begin{align}\label{eq:variation_constant}
    X_t = x_0(t) - \int_0^t R_\lambda(t-s)x_0(s)\ds + \frac{1}{\lambda}\int_0^t R_\lambda(t-s)\phi(s)\D W_s.
\end{align}
This formulation exploits the mean-reverting property to essentially replace the kernel $\lambda K$ by a kernel~$R_\lambda$ that decays more rapidly. As an application of Proposition \ref{prop:BDG_convo_uniform}, this allows to derive a uniform-in-time bound.
\begin{corollary}\label{coro:uniform_time}
    Assume that $R_\lambda$ is completely monotone with measure $\mu_\lambda$. If there exists $p> 2$ such that~$\int_0^\infty \EE\abs{\phi(s)}^p\ds<\infty$ and
    \begin{align}\label{eq:condition_bound_infty}
        C_{p,x,\lambda,\mu}:=\sup_{t\ge0} \abs{x_0(t)} +\lim_{t\to\infty} \abs{\int_0^t R_\lambda(t-s)x_0(s)\ds}+\left( \int_0^\infty x^{\frac{2-p}{2p}}\mu_\lambda(\D x)\right)^p<\infty,
    \end{align}
    then there is another constant~$C_p>0$ depending only on~$p$ such that
    \begin{align*}
        \EE\left[\sup_{t\ge0} \abs{X_t}^p\right] \le C_p C_{p,x,\lambda,\mu_\lambda} \int_0^\infty \EE\abs{\phi(s)}^p\ds.
    \end{align*}
    In the particular case $d=m=1$, $x_0$ bounded and $K(t)=t^{H-\half}/\Gamma(H+\half)$ for $H\in(0,\half)$, $R_\lambda$ is completely monotone and satisfies~\eqref{eq:condition_bound_infty} for any~$p>1/H$.
\end{corollary}
\begin{proof}
    We study the expression of $X$ obtained in \eqref{eq:variation_constant}. Firstly, we have
    \begin{align*}
        \sup_{t\ge0} \abs{x_0(t)-\int_0^t R_\lambda(t-s)x_0(s)\ds} \le  \sup_{t\ge0}\abs{x_0(t)}+ \lim_{t\to\infty} \abs{\int_0^t R_\lambda(t-s)x_0(s)\ds}, 
    \end{align*}
    which is finite by assumption. Then we turn our attention to the stochastic Volterra integral~$\int_0^t R_\lambda(t-s)\phi(s)\D W_s$ to which we can apply Proposition~\ref{prop:BDG_convo_uniform} by virtue of Condition~\eqref{eq:condition_bound_infty}.
  
    Focusing now on the case $K(t)=t^{H-\half}/\Gamma(H+\half)$, $H\in(0,\half)$, the resolvent of $\lambda K$ reads 
    \begin{align*}
        R_\lambda(t)
        =\lambda t^{H-\half} E_{H+\half,H+\half}(-\lambda t^{H+\half}),
    \end{align*}
    where $E_{\alpha,\beta}(z)=\sum_{n\ge0} \frac{z^n}{\Gamma(\alpha n+\beta)}$ is the Mittag-Leffler function. First, Proposition \ref{prop:BDG_convo_uniform} states that $R_\lambda$ is completely monotone with a measure $\mu_\lambda$ that satisfies $\int_0^\infty x^{\frac{2-p}{2p}}\mu_\lambda(\D x)<\infty$ for any $p>1/H$. 
    Since $x_0$ is bounded and $R_\lambda$ is positive we have
    \begin{align}\label{eq:int_Rlambda}
        \int_0^t R_\lambda(t-s) x_0(s) \ds 
        \le \sup_{t\ge0}\abs{x_0(t)} 
        \int_0^t R_\lambda(s)\ds
        =\sup_{t\ge0}\abs{x_0(t)}  \lambda t^{H+\half} E_{H+\half,H+\frac32}(-\lambda t^{H+\half}).
    \end{align}
    It is proved in \cite[Lemma 2.5]{wang2018note} (see also \cite{gorenflo2002computation} but without a proof) that $E_{\alpha,\alpha+1}(-\lambda t^{\alpha})\le C_{\lambda,\alpha} (t^{-2\alpha}+t^{-\alpha})$ for any $\alpha\in(0,1]$ and where $C_{\lambda,\alpha}>0$. Therefore the limit  of \eqref{eq:int_Rlambda} as $t$ goes to infinity is finite. 
\end{proof}


\section{Formerly known BDG inequalities}\label{sec:nonconv}
The BDG inequalities presented in this section are not new; however, they are likely to be unfamiliar to certain segments of the community who may find them useful. We believe it would be beneficial to compile them in this note for easier reference. 
Moreover, they apply to kernels which are not of convolution type. 


\subsection{Decreusefond's inequalities}
This section intends to shed some light on Decreusefond's work on stochastic Volterra integrals~\cite{decreusefond2002regularity}, where the first BDG inequality for such integrals can be traced back to. The results of this paper are stated on the interval~$[0,1]$ and in the one-dimensional case~$d=m=1$, hence this is how we present them. We write $\Ll^p$ in place of $\Ll^p_1$ in this section. 

We need a couple of definitions to start with. For any~$\alpha\in(0,1)$ and~$f\in \Ll^1$, define the fractional integral
\begin{align*}
    (I^\alpha f)(t) = \frac{1}{\Gamma(\alpha)} \int_0^t (t-s)^{\alpha-1} f(s)\ds,\quad t\in[0,1].
\end{align*}
Further introduce the space~$\Ii^\alpha_p:=I^\alpha(\Ll^p)$ for all~$p\ge 1$. Denoting~$I^{-\alpha}$ the inverse map of~$I^\alpha$, we equip this space with the norm~$\norm{f}_{\Ii^\alpha_p}:=\norm{I^{-\alpha}f}_{\Ll^p}$. For a kernel~$K:[0,1]^2\to\RR$ we denote by~$V_K:\Ll^1\to\RR$ the linear map~$V_K f(t)=\int_0^1 K(t,s)f(s)\ds$. Finally we set~$\theta(x)=\frac{2x}{2-x}$ for all~$x\le 2$.
\begin{theorem}\cite[Theorem 3.1]{decreusefond2002regularity} \label{thm:Decreusefond}
    Assume that there exists~$\alpha>0$ and~$\eta\le 2$ such that~$V_K$ is continuous from~$\Ll^2$ to~$\Ii^{\alpha+\half}_2$ and from~$\Ll^\eta$ to~$\Ii^{\alpha}_{\theta(\eta)}$. Let~$p=\theta(\eta)$ and assume furthermore that~$\phi$ is a previsible process satisfying for all~$t\in[0,1]$
    \begin{align*}
        \int_0^1 \EE\abs{\phi(s)}^p \ds + \int_0^1 K(t,s)^2\EE\abs{\phi(s)}^2 \ds <\infty.
    \end{align*}
    Then~$\left\{\int_0^t K(t,s)\phi(s)\D W_s,\, t\in[0,1]\right\}$ has a version which belongs to~$\bigcap_{\gamma<\alpha} \Ii^{\gamma}_p$ and, for any~$\gamma<\alpha$,
    \begin{align*}
        \EE\norm{\int_0^\cdot K(\cdot,s)\phi(s)\D W_s}_{\Ii^\gamma_p}^p \le c_{\gamma,\eta}^p \int_0^1 \EE\abs{\phi(s)}^p\ds.
    \end{align*}
    The constant is given by~$c_{\gamma,\eta}=\sup_{g:\norm{g}_{\Ll^\eta}=1} \norm{(I^{-\gamma} \circ V_K)g}_{\Ll^{p}}$.
\end{theorem}
This immediately leads to a BDG inequality since~$\Ii^{\gamma}_p$ is continuously embedded in the space of $(\gamma-1/p)$-H\"older continuous functions for any~$\gamma>1/p$. 
\begin{corollary}
    Under the same assumptions as Theorem~\ref{thm:Decreusefond} and for any $\gamma<\alpha$, there is another constant $C_{p,\gamma}>0$ such that 
    \begin{align}\label{eq:BDG_decreusefond}
        \EE\left[\sup_{t\in[0,1]} \abs{\int_0^t K(t,s)\phi(s)\D W_s}^p\right]\le C_{p,\gamma} \int_0^1 \EE\abs{\phi(s)}^p\ds.
    \end{align}
\end{corollary}
Decreusefond then proceeds to verify that these assumptions are satisfied for two choices of kernels. For~$H\in(0,1)$, he considers the power-law kernel~$J_H(t,s)=\Gamma(H+\half)^{-1} (t-s)^{H-\half}\one_{t>s}$ and
\begin{align*}
    K_H(t,s)=J_H(t,s) F(H-1/2,1/2-H,H+1/2,1-t/s),
\end{align*}
where~$F$ is the Gauss hypergeometric function. The latter kernel gives rise to the fractional Brownian motion of Mandelbrot and Van Ness, as shown in~\cite[Corollary 3.1]{decreusefond1999stochastic}.
For those kernels it is proved, in \cite{decreusefond2002regularity} in Theorem 4.1 and 4.2 respectively, that the conditions of Theorem~\ref{thm:Decreusefond} hold for any~$p\ge2$ and any~$H>1/p$. Remark 4.1 of that paper is the first instance (as far as we are aware) of a BDG inequality for stochastic Volterra integrals.


\subsection{Kolmogorov's continuity criterion}
This will probably not come as a big surprise to Volterra experts that Kolmogorov's continuity theorem can be used to derive a BDG inequality for stochastic Volterra integrals. This was actually proved in \cite[Lemma 3.4]{zhang2010stochastic} in the much more general context of Banach space valued processes. We adapt their notations and present this result in the context of this paper. 
\begin{proposition}
    Let $K\in L^2([0,T]^2;\RR^{d\times d})$ and assume there exist~$\gamma>2,\,\beta>0$ and~$C_K>0$ such that, for all~$0\le s<t\le T$,
    \begin{align}\label{eq:kernel_nonconvo_condition}
        \int_s^t \abs{K(t,r)}^\gamma\dr + \int_0^s \abs{K(t,r)-K(s,r)}^\gamma\dr\le C_K(t-s)^{\beta}.
    \end{align}
    Let~$p>\max\left(\frac{2\gamma}{\gamma-2},\frac{\gamma}{\beta}\right)$ and~$\phi$ be an~$\RR^{d\times m}$-valued previsible process such that $\int_0^T \EE\abs{\phi(s)}^p\ds < \infty$. 
    Then there are constants $C_0,C_1>0$ such that
    \begin{align*}
        \EE\abs{\int_0^t K(t,s)\phi(s)\D W_s-\int_0^{t'} K(t',s)\phi(s)\D W_s}^p \le C_0 \abs{t-t'}^{\frac{\beta p}{\gamma}} \int_0^T \EE\abs{\phi(s)}^p\ds,
    \end{align*}
    and by Kolmogorov's continuity theorem it holds
    \begin{align}\label{eq:BDG_nonconv}
    \EE\left[\sup_{t\in[0,T]} \abs{\int_0^t K(t,s)\phi(s)\D W_s}^p\right]\le C_1 \int_0^T \EE\abs{\phi(r)}^p\dr.
\end{align}
\end{proposition}
The condition $p>\frac{\gamma}{\beta}$ ensures that $\frac{\beta p}{\gamma}>1$ for the Kolmogorov continuity criterion to apply. On the other hand, it is also required that~$p>\frac{2\gamma}{\gamma-2}$ (the same condition as in Theorem~\ref{thm:BDG_convolution}) for Jensen's inequality to hold in the following estimate
\begin{align*}
    \EE\left(\int_s^t \abs{K(t,r)}^2 \abs{\phi(r)}^2 \dr \right)^{p/2} 
    \le  \norm{K}_{\Ll^\gamma([s,t])}^p \EE\left(\int_s^t \abs{\phi(r)}^{\frac{2\gamma}{\gamma-2}}\dr \right)^{\frac{p}{2}\frac{\gamma-2}{\gamma}}
    \le C (t-s)^{\frac{\beta p}{\gamma}} \int_0^T \EE\abs{\phi(r)}^p\dr.
\end{align*}
The power-law kernel~$K(t,s)=(t-s)^{H-\half},\,t>s,$ satisfies condition~\eqref{eq:kernel_nonconvo_condition} for all~$\gamma\in(2,\frac{2}{1-2H})$ and~$\beta=\gamma(H-\half)+1$ if~$H\in(0,\half)$. With these variables, $\frac{\gamma}{\beta}<\frac{2\gamma}{\gamma-2}$ if and only if~$\gamma<\frac{2}{1-H}$. Since~$\gamma\mapsto \frac{\gamma}{\gamma(H-\half)+1}$ is increasing while~$\gamma\mapsto \frac{2\gamma}{\gamma-2}$ is decreasing, the minimum of~$\max\big(\frac{\gamma}{\gamma(H-\half)+1},\frac{2\gamma}{\gamma-2}\big)$ is attained at~$\gamma^\ast=\frac{2}{1-H}$ and takes the value~$p^\ast=2/H$. 
For comparison, recall that the convolution BDG inequality~\eqref{eq:BDG_convolution} and Decreusefond's inequality~\eqref{eq:BDG_decreusefond} hold for any~$p>1/H$.


\bibliographystyle{abbrv}
\bibliography{bib}
\end{document}